\documentclass[smallextended,nospthms]{svjour3}       
\smartqed  
\usepackage{mathptmx}      

\usepackage{amssymb,amsmath,enumerate,amsthm}

\theoremstyle{plain}
\newtheorem{theorem}{Theorem}[section]
\newtheorem{corollary}[theorem]{Corollary}

\theoremstyle{remark}

\newcommand{\CC}{{\mathbb C}}
\newcommand{\DD}{{\mathbb D}}

\newcommand{\cD}{{\mathcal D}}
\newcommand{\cH}{{\mathcal H}}
\DeclareMathOperator{\hol}{\mathrm Hol}

\journalname{Journal}

\begin{document}

\title{Linear polynomial approximation schemes in Banach holomorphic function spaces
\thanks{First author supported by a grant from NSERC. \\Second author supported by grants from NSERC and the Canada Research Chairs program.}}

\titlerunning{Linear polynomial approximation schemes}        

\author{Javad Mashreghi  \and Thomas Ransford}

\authorrunning{J. Mashreghi \and T. Ransford} 

\institute{J. Mashreghi \at
              D\'epartement de math\'ematiques et de statistique, Universit\'e Laval, Qu\'ebec (QC), G1V 0A6, Canada \\
                                 \and
           T. Ransford \at
              D\'epartement de math\'ematiques et de statistique, Universit\'e Laval, Qu\'ebec (QC), G1V 0A6, Canada\\
              Tel.: +1-418-656-2131\\
              \email{ransford@mat.ulaval.ca}   
}

\date{Received: date / Accepted: date}

\maketitle

\begin{abstract}
Let $X$ be a Banach holomorphic function space on the unit disk.
A \emph{linear polynomial approximation scheme} for  $X$ 
is a sequence of bounded linear operators $T_n:X\to X$ with the property that, for each $f\in X$, 
the functions $T_n(f)$ are polynomials converging to $f$ in the norm of the space.
We completely characterize those spaces $X$ that admit a linear polynomial approximation scheme.
In particular, we show that it is \emph{not} sufficient merely that polynomials be dense in $X$.

\end{abstract}


\section{Introduction}\label{S:intro}

We denote by $\DD$ the open unit disk, and by $\hol(\DD)$ the Fr\'echet space of holomorphic functions on $\DD$,
equipped with the topology of uniform convergence on compact sets. 

A \emph{Banach holomorphic function space on $\DD$} 
is a Banach space $X$ that is a subset of $\hol(\DD)$,
such that the inclusion map $X\hookrightarrow\hol(\DD)$ is continuous.

In many such spaces $X$, the polynomials are dense. 
This is most usually proved by a direct construction.
For example, if $X$ is the Hardy space $H^2$, then, given $f\in H^2$, 
we can approximate $f$ by the partial sums  $s_n(f)$ of its Taylor expansion. Clearly the $s_n(f)$ are polynomials,
and,  from the very definition of the $H^2$-norm, we have $\|s_n(f)-f\|_{H^2}\to0$ as $n\to\infty$.

However, there are spaces $X$ where this simple procedure breaks down, 
even though the polynomials are dense.
A well-known example is the disk algebra $A(\DD)$. 
There exists $f\in A(\DD)$ such that $s_n(f)$ does not converge to $f$ in the norm of $A(\DD)$,
and the sequence $\|s_n(f)\|_{A(\DD)}$ is even unbounded.
This is proved using a slight variant of the argument that establishes the theorem of du Bois-Reymond that
there exists a continuous function on the unit circle whose Fourier series diverges
(see e.g.\ \cite[Chapter~II, Theorem~2.1]{Ka76}).
In this case, there is an elegant way around the problem.
Let 
\begin{equation}\label{E:Cesaro}
\sigma_n(f):=\frac{1}{n+1}\sum_{k=0}^n s_k(f).
\end{equation}
Then $\sigma_n(f)$ is still a polynomial, and $\|\sigma_n(f)-f\|_{A(\DD)}\to0$ as $n\to\infty$.
This is essentially Fej\'er's theorem.
The same technique also works in certain other spaces $X$,
notably weighted Dirichlet spaces $\cD(\omega)$ with superharmonic weights $\omega$ (see \cite{MR19}).

However, there are  spaces where even this method fails.
Given a holomorphic function $b:\DD\to\DD$, 
we denote by $\cH(b)$ the de Branges--Rovnyak space associated to $b$.
We shall not give the precise definition of these spaces here, but remark simply that
each $\cH(b)$ is a Hilbert space continuously included in $H^2$.
It is known that polynomials are dense in $\cH(b)$ if and only if
$\int_0^{2\pi}\log(1-|b^*(e^{i\theta})|^2)\,d\theta>-\infty$
(where $b^*$ denotes the radial limit function of $b$).
It was shown in \cite{EFKMR16} that 
there exists a de Branges--Rovnyak space $\cH(b)$ in which polynomials are dense
and a function $f\in\cH(b)$ such that neither $s_n(f)$ nor $\sigma_n(f)$  converge to $f$ in the norm of $\cH(b)$
(and in fact both sequences are unbounded).

The article \cite{EFKMR16} also describes a constructive method for polynomial approximation in
those spaces $\cH(b)$ where polynomials are dense. 
However, the method depends not only on $b$ but on the function being approximated, and, as a result,
it is highly nonlinear. This leaves open the problem of whether there is a linear method that works,
for example, replacing the Ces\`aro means in \eqref{E:Cesaro} by some other summability method.
To formulate this idea more precisely, we introduce the following definition.

Given a Banach holomorphic function space $X$ on $\DD$,
a \emph{linear polynomial approximation scheme} for $X$ is a sequence
of continuous linear operators $T_n:X\to X$ such that, for each $f\in X$, the functions $T_n(f)$
are polynomials and $\|T_n(f)-f\|_X\to0$ as $n\to\infty$.

Thus, our problem becomes: 
does every de Branges--Rovnyak space $\cH(b)$ in which polynomials are dense admit a
linear polynomial approximation scheme? 
This is a very special case of a much wider question, namely,
which Banach holomorphic function spaces on $\DD$ admit a linear
polynomial approximation scheme? 
Our first theorem furnishes a complete answer to this question.

\begin{theorem}\label{T:BAP}
A Banach holomorphic function space $X$ on $\DD$
admits a linear polynomial approximation scheme
if and only if $X$ contains a dense subspace of polynomials
and has the bounded approximation property.
\end{theorem}

We recall that a Banach space $(Y,\|\cdot\|)$ has the \emph{approximation property} (AP) if,
given a compact subset $K$ of $Y$ and $\epsilon>0$, 
there is a bounded finite-rank operator $T:Y\to Y$ such that
\[
\|Ty-y\|\le \epsilon \quad(y\in K).
\]
If, in  addition, the operator $T$ can chosen so that $\|T\|\le M$, 
where $M$ is a constant independent of $K$ and $\epsilon$,
then $Y$ is said to have the \emph{bounded approximation property} (BAP).

It is well known that every Banach space with a Schauder basis has BAP.
In particular, separable Hilbert spaces have BAP. Thus we obtain the following corollary,
which answers our earlier question about the de Branges--Rovnyak spaces $\cH(b)$.

\begin{corollary}\label{C:Hilbert}
Let $H$ be a Hilbert holomorphic function space on $\DD$.
Then $H$ admits a linear polynomial approximation scheme if and only if it contains a dense
subspace of polynomials.
\end{corollary}

There exist separable Banach spaces without BAP (and even without AP).
The first such example was given by Enflo \cite{En73}.
It is now known that several naturally occurring Banach spaces fail to possess BAP,
for example, certain closed subspaces of $c_0$ and $\ell^p~(p\ne2)$
(see \cite{Da73} and \cite{Sz78}).

This raises the question as to whether there exist
Banach holomorphic function spaces on $\DD$ without BAP.
Our second theorem answers this question in the affirmative, 
by showing that any separable, infinite-dimensional, complex Banach space 
can be manifested as a Banach holomorphic function space on $\DD$ in which the polynomials are dense.
The result even permits a certain amount of control on the norm.

\begin{theorem}\label{T:model}
Let $Y$ be a separable, infinite-dimensional, complex Banach space,
and let $(\alpha_n)_{n\ge0}$ be a positive sequence such that 
$\lim_{n\to\infty}\alpha_n^{1/n}=1$.
Then there exists a Banach holomorphic function space $X$ on $\DD$ such that:
\begin{enumerate}[(i)]
\item $X$ is isometrically isomorphic to $Y$,
\item $X$ contains all functions holomorphic in a neighborhood of $\overline{\DD}$,
\item polynomials are dense in $X$,
\item $\|z^n\|_X=\alpha_n$ for all $n\ge0$.
\end{enumerate}
\end{theorem}

Combining Theorems~\ref{T:BAP} and \ref{T:model}, we obtain the following corollary.

\begin{corollary}
There exists a Banach holomorphic function space $X$ on~$\DD$, 
containing the polynomials as a dense subspace, 
but not admitting any linear polynomial approximation scheme.
\end{corollary}

\begin{proof}
Let $Y$ be a separable, complex Banach space without BAP, 
and let $X$ be a Banach holomorphic function space on $\DD$ 
satisfying the properties listed in Theorem~\ref{T:model}.
Then polynomials are dense in $X$, but $X$ does not have BAP and so, by Theorem~\ref{T:BAP}, 
it does not admit a linear polynomial approximation scheme.
\end{proof}


\section{Proof of Theorem~\ref{T:BAP}}\label{S:proofthm1}

We shall establish the following abstract version of Theorem~\ref{T:BAP}.
The original version follows upon taking
$Y=X$ and $Z=X\cap\{\text{polynomials}\}$.

\begin{theorem}\label{T:abstract}
Let $Y$ be a Banach space and let $Z$ be a subspace of $Y$ admitting a countable Hamel basis.
Then the following statements are equivalent.
\begin{enumerate}[(i)]
\item There exists a sequence of bounded linear maps $T_n:Y\to Z$ such that, for each $y\in Y$,
we have $\|T_ny-y\|\to0$ as $n\to\infty$.
\item The subspace $Z$ is dense in $Y$, and $Y$ has the bounded approximation property.
\end{enumerate}
\end{theorem}

\begin{proof}
(i) $\Rightarrow$ (ii):
Suppose that (i) holds. Then, obviously, $Z$ is dense in $Y$. 
To prove that (ii) holds, it remains to show that $Y$ has BAP.

Let $T_n:Y\to Z$ be  bounded linear maps such that $\|T_ny-y\|\to0$ for each $y\in Y$.
Then $M:=\sup_n\|T_n\|<\infty$, by the Banach--Steinhaus theorem.
Also, we claim  that each $T_n$ is of finite rank.
Indeed, since $Z$ has a countable Hamel basis, we can write $Z=\cup_{k\ge1}F_k$,
where the $F_k$ are finite-dimensional subspaces of $Y$, hence closed in $Y$.
Then, for each $n$, we have $Y=T_n^{-1}(Z)=\cup_{k\ge1}T_n^{-1}(F_k)$, a countable union of closed sets,
so, by the Baire category theorem, 
at least one of the sets $T_n^{-1}(F_k)$ has non-empty interior, $T_n^{-1}(F_{k_0})$ say.
As $T_n^{-1}(F_{k_0})$ is a subspace of $Y$,
this forces $T_n^{-1}(F_{k_0})=Y$, 
which in turn implies that $T_n(Y)\subset F_{k_0}$.
Thus $T_n$ is of finite rank, as claimed.

Now let $K$ be a compact subset of $Y$ and let $\epsilon>0$.
Cover $K$ by a finite set of closed balls $B(y_j,\epsilon/(M+1))$.
Choose $N$ large enough so that $\|T_N(y_j)-y_j\|\le\epsilon$ for each $j$.
Given $y\in K$, we may find  $j$ so that $y\in B(y_j,\epsilon/(M+1))$, and then
\begin{align*}
\|T_N(y)-y\|
&\le \|T_N(y-y_j)\|+\|T_N(y_j)-y_j\|+\|y_j-y\|\\
&\le M\frac{\epsilon}{M+1}+\epsilon+\frac{\epsilon}{M+1}=2 \epsilon.
\end{align*}
Combined with the fact that $\|T_N\|\le M$, this shows that $Y$ has BAP.

\smallskip

(ii) $\Rightarrow$ (i):
Assume that (ii) holds. As $Y$ has BAP,
there exists a constant $M$ such that,
for each compact subset $K$ of $Y$ and $\epsilon>0$, 
we may find a finite-rank operator $T:Y\to Y$ satisfying
\begin{equation}\label{E:AP}
\|T\|\le M\quad\text{and}\quad \|Ty -y\|\le \epsilon \quad(y\in K).
\end{equation}
We claim that $T$ may be chosen so that, in addition, $T(Y)\subset Z$.
Indeed, as $T$ is finite-rank and $Z$ is dense in $Y$, 
we can find a finite-rank operator $T'$ such that 
$T'(Y)\subset Z$ and $\|T-T'\|$ is as small as we please.
Replacing $T$ by $T'$ and adjusting the constants appropriately,
the claim is justified.

The hypotheses in (ii) certainly imply that $Y$ is separable.
Let $(y_n)_{n\ge1}$ be a dense sequence in $Y$.
By we what have just proved, applied with $K:=\{y_1,\dots,y_n\}$,
for each $n\ge1$ there exists an operator $T_n:Y\to Z$ such that
\[
\|T_n\|\le M \quad\text{and}\quad \|T_n(y_j)-y_j\|\le 1/n \quad(j=1,\dots,n).
\]
We claim that $\lim_{n\to\infty}\|T_n(y)-y\|=0$ for each $y\in Y$. 
Indeed, given $\epsilon>0$, choose $y_{n_0}$ so that $\|y-y_{n_0}\|<\epsilon/(M+1)$.
Then, for all $n\ge n_0$, we have
\begin{align*}
\|T_n(y)-y\|
&\le \|T_n(y-y_{n_0})\|+\|T_n(y_{n_0})-y_{n_0}\|+\|y_{n_0}-y\|\\
&\le M\frac{\epsilon}{M+1}+\frac{1}{n}+\frac{\epsilon}{M+1}=\epsilon+\frac{1}{n}.
\end{align*}
Thus $\|T_n(y)-y\|\le 2\epsilon$ for all $n$ large enough, thereby establishing the claim.
This shows that (i) holds, and completes the proof of the theorem.
\end{proof}


\section{Proof of Theorem~\ref{T:model}}\label{S:proofthm2}

A separable Banach space that fails to have the bounded approximation property 
cannot have a Schauder basis.
However, it may admit certain weaker types of basis. 
Our proof of Theorem~\ref{T:model} relies on one notion of this kind.

Let $(Y,\|\cdot\|)$ be a Banach space and denote by $Y^*$ its dual.
A \emph{biorthogonal system} for $Y$ is a  sequence $(e_n,e_n^*)$ in $Y\times Y^*$
such that
\[
e_j^*(e_k)=
\begin{cases}
1, &j=k,\\ 0, &j\ne k.
\end{cases}
\]
It is called \emph{fundamental} if the sequence $(e_n)$ spans a dense subspace of $Y$,
and \emph{total} if the sequence $(e^*_n)$ spans a weak*-dense subspace of $Y^*$.
A biorthogonal system that is both fundamental and total is called a \emph{Markushevich basis} for $Y$. 

The first part of the following result is due to Markushevich \cite{Ma43}, 
and the second part is due to Ovsepian and Pe\l czy\'nski \cite{OP75}.
A proof can also be found in  \cite[Proposition~1f3 and Theorem~1f4]{LT77}.

\begin{theorem}\label{T:Markushevich}
Every separable, infinite-dimensional Banach space admits a Markushevich basis $(e_n,e^*_n)$.
This basis may be chosen so that also $\sup_n\|e_n\|\|e^*_n\|<\infty$.
\end{theorem}

\begin{proof}[Proof of Theorem~\ref{T:model}]
Given $Y$ as in the statement of the theorem, 
let $(e_n,e^*_n)_{n\ge0}$ be a Markushevich basis for $Y$
with $M:=\sup_{n\ge0}\|e_n\|\|e^*_n\|<\infty$. 
Multiplying the $e_n$ by appropriate positive constants, 
and dividing the $e_n^*$ by the same constants,
we may further suppose that $\|e_n\|=\alpha_n$ for all $n\ge0$, 
whence  $\|e_n^*\|\le M/\alpha_n$ for all $n$.

For $y\in Y$, define
\[
(Jy)(z):=\sum_{n\ge0} e_n^*(y)z^n \quad(z\in \DD).
\]
Note that
\[
\limsup_{n\to\infty}|e_n^*(y)|^{1/n}
\le \limsup_{n\to\infty}\Bigl(\frac{M\|y\|}{\alpha_n}\Bigr)^{1/n}
=\lim_{n\to\infty}\alpha_n^{-1/n}=1,
\]
so the power series defining $Jy$ has radius of convergence at least~$1$,
in other words, $Jy\in\hol(\DD)$. Further,
given $r\in(0,1)$, we have
\[
\sup_{|z|\le r}|(Jy)(z)|
\le \sum_{n\ge0}|e_n^*(y)|r^n\le \|y\|\sum_{n\ge0}Mr^n/\alpha_n=C_r\|y\|,
\]
where $C_r:=\sum_{n\ge0}Mr^n/\alpha_n<\infty$, by the $n$-th root test.
Lastly, if $Jy=0$, then $e_n^*(y)=0$ for all $n$, which implies that $y=0$, 
since the sequence $(e_n^*)$ is total. 
To summarize: we have shown that $J:Y\to\hol(\DD)$ is a continuous, linear injection.

Define $X:=J(Y)$, with the norm inherited from $Y$. 
Clearly $X$ is a Banach holomorphic function space on $\DD$ that is isometrically isomorphic to~$Y$.
Also, since the sequence $(e_n)$ is fundamental in $Y$ and $J$ maps $e_n$ to $z^n$, 
it follows that the polynomials form a dense subspace of $X$.
By construction, we have $\|z^n\|_X=\|e_n\|=\alpha_n$ for all $n\ge0$.
Finally, if $f$ is a function holomorphic on a neighbourhood of $\overline{\DD}$,
then its Taylor series has radius of convergence $R>1$. Together with the fact that
$\lim_{n\to\infty}\|z^n\|_X^{1/n}=\lim_{n\to\infty}\alpha_n^{1/n}=1$,
this implies that the Taylor series of $f$ converges absolutely in the norm of $X$, whence $f\in X$.
Thus $X$ has all the properties (i)--(iv) listed in the statement of the theorem.
\end{proof}


\section{Concluding remarks}\label{S:conclusion}

(1) The proof of Theorem~\ref{T:BAP} shows that, 
if $T_n:X\to X$ is a linear polynomial approximation scheme for $X$,
then $T_n(f)$ is a polynomial of degree at most $N(n)$, 
where $N(n)$ does not depend on $f$. 
By repeating some terms, if necessary, 
we can even arrange that $\deg(T_n(f))\le n$ for all $n$ and for all $f\in X$.
The examples $T_n(f):=s_n(f)$ and $T_n(f):=\sigma_n(f)$ mentioned in the introduction 
already satisfy this inequality.

(2) Although our primary interest in this article has been in subspaces of $\hol(\DD)$,
the techniques developed are easily adapted to function spaces on other domains,
for example, balls and polydisks in $\CC^n$. Also, instead of considering approximation by
polynomials, we can apply our techniques to other approximating subspaces. An interesting example
is when $X\subset\hol(\CC^+)$ (where $\CC^+$ denotes the upper half-plane)
and we approximate by linear combinations of  $\phi_n(z):=(z-i)^n/(z+i)^{n+1}~(n\ge0)$.

(3) Although Theorem~\ref{T:BAP} completely settles the question of the existence
of linear polynomial approximation schemes $T_n:X \to X$, 
and even implicitly furnishes a method for constructing them,
it still leaves open the problem of finding a workable formula for $T_n$.
Finding such a formula will surely depend on the space $X$.

An interesting test case is that of the de Branges--Rovnyak spaces $\cH(b)$.
Given  a space $\cH(b)$ in which the polynomials are dense,
can we find an explicit triangular array of complex numbers 
$(a_{nk})_{0\le k\le n<\infty}$ such that
\[
T_n(f):=\sum_{k=0}^n a_{nk}s_k(f)
\]
defines a linear polynomial approximation scheme for $\cH(b)$?
Can we even find such an array that works simultaneously
for all such spaces $\cH(b)$?

\bibliographystyle{spmpsci}      
\bibliography{biblist}

\end{document}